\documentclass[12pt,reqno]{amsart}

\usepackage{amssymb}
\usepackage{mathtools}

\usepackage{hyperref}

\newtheorem{theorem}{Theorem}[section]
\newtheorem{proposition}[theorem]{Proposition}
\newtheorem{lemma}[theorem]{Lemma}
\newtheorem{corollary}[theorem]{Corollary}
\newtheorem{question}[theorem]{Question}

\theoremstyle{definition}

\numberwithin{equation}{section}

\begin{document}

\title[Moduli of vector bundles on a curve and opers]{Moduli spaces of vector bundles on a 
curve and opers}

\author[I. Biswas]{Indranil Biswas}

\address{School of Mathematics, Tata Institute of Fundamental Research,
Homi Bhabha Road, Mumbai 400005, India}

\email{indranil@math.tifr.res.in}

\author[J. Hurtubise]{Jacques Hurtubise}

\address{Department of Mathematics, McGill University, Burnside
Hall, 805 Sherbrooke St. W., Montreal, Que. H3A 2K6, Canada}

\email{jacques.hurtubise@mcgill.ca}

\author[V. Roubtsov]{Vladimir Roubtsov}

\address{UNAM, LAREMA UMR 6093 du CNRS, Universit\'e d'Angers, 49045 Cedex 01,
Angers, France}

\email{vladimir.roubtsov@univ-angers.fr}

\subjclass[2010]{14H60, 16S32, 14D21, 53D30}

\keywords{Stable bundle, moduli space, connection, opers}

\date{}

\begin{abstract}
Let $X$ be a compact connected Riemann surface of genus $g$, with $g\, \geq\,2$, and let 
$\xi$ be a holomorphic line bundle on $X$ with $\xi^{\otimes 2}\,=\, {\mathcal O}_X$. Fix 
a theta characteristic $\mathbb L$ on $X$. Let ${\mathcal M}_X(r,\xi)$ be the moduli space 
of stable vector bundles $E$ on $X$ of rank $r$ such that $\bigwedge^r E\,=\, \xi$ and $H^0(X,\, 
E\otimes{\mathbb L})\,=\, 0$. Consider the quotient of ${\mathcal M}_X(r,\xi)$ by the 
involution given by $E\, \longmapsto\, E^*$. We construct an algebraic morphism from this 
quotient to the moduli space of ${\rm SL}(r,{\mathbb C})$ opers on $X$. Since $\dim 
{\mathcal M}_X(r,\xi)$ coincides with the dimension of the moduli space of ${\rm 
SL}(r,{\mathbb C})$ opers, it is natural to ask about the injectivity and surjectivity of 
this map.
\end{abstract}

\maketitle

\section{Introduction}\label{sec1}

Opers were introduced by Beilinson and Drinfeld \cite{BD1}, \cite{BD2}.
Our aim here is to construct ${\rm SL}(n, {\mathbb C})$ opers from stable vector bundles of degree zero.
While a stable vector bundle of degree zero has a unique unitary flat connection, unitary connections are
never an oper.

Take a compact connected Riemann surface of genus $g$, with $g\, \geq\, 2$, and fix a theta characteristic
$\mathbb L$ on $X$. Let ${\mathcal M}_X(r)$ be the moduli space of stable vector bundles $E$ of rank $r$ and
degree zero on $X$ such that $H^0(X,\, E\otimes {\mathbb L})\,=\, 0$. For $i\,=\, 1,\, 2$, the projection
$X\times X\,\longrightarrow\, X$ to the $i$-th factor is denoted by $p_i$. The diagonal divisor in $X\times
X$ is denoted by $\Delta$; it is identified with $X$ by $p_i$. For any $E\, \in\,
{\mathcal M}_X(r)$, there is a unique section
$$
{\mathcal A}_E\, \in\, H^0(X\times X,\, p^*_1 (E\otimes{\mathbb L})\otimes p^*_2
(E^*\otimes{\mathbb L})\otimes{\mathcal O}_{X\times X}(\Delta))
$$
whose restriction to $\Delta$ is ${\rm Id}_E$ (using the identification of $\Delta$ with $X$).

Using ${\mathcal A}_E$, we construct an ${\rm SL}(n,{\mathbb C})$ oper on $X$ for every 
$n\, \geq\, 2$; see Theorem \ref{thm1} and Proposition \ref{prop1}. Related
construction of opers from vector bundles were carried out in \cite{BB}.

Let ${\rm Op}_X(n)$ denote the moduli space of ${\rm SL}(n,{\mathbb C})$ opers on $X$. The above mentioned map
$$
{\mathcal M}_X(r) \, \longrightarrow\, {\rm Op}_X(n)
$$
factors through the quotient of ${\mathcal M}_X(r)$ by the involution $\mathcal I$ defined by $E\, \longmapsto\, E^*$.

Fix a holomorphic line bundle $\xi$ on $X$ such that $\xi^{\otimes 2}\,=\, {\mathcal O}_X$. Let
$$
{\mathcal M}_X(r, \xi)\, \subset\, {\mathcal M}_X(r)
$$
be the subvariety defined by the locus of all $E$ such that $\bigwedge^r E\,=\, \xi$. We have
$$
\dim {\mathcal M}_X(r,\xi)/{\mathcal I}\,=\, (r^2-1)(g-1)\,=\, \dim {\rm Op}_X(r)\, .
$$
We end with a question (see Question \eqref{q1}).

\section{Vector bundles with trivial cohomology}

Let $X$ be a compact connected Riemann surface of genus $g$, with $g\, \geq\, 2$. The holomorphic
cotangent bundle of $X$ will be denoted by $K_X$. Fix a theta characteristic $\mathbb L$ on
$X$. So, $\mathbb L$ is a holomorphic line bundle on $X$ of degree $g-1$, and
${\mathbb L}\otimes \mathbb L$ is holomorphically isomorphic to $K_X$.

For any $r\, \geq\, 1$, let $\widetilde{\mathcal M}_X(r)$ denote the moduli space of stable vector
bundles on $X$ of rank $r$ and degree zero. It is an irreducible smooth complex quasiprojective
variety of dimension $r^2(g-1)+1$. Let
\begin{equation}\label{e1}
{\mathcal M}_X(r)\, \subset\, \widetilde{\mathcal M}_X(r)
\end{equation}
be the locus of all vector bundles $E\, \in\, \widetilde{\mathcal M}_X(r)$ such that $H^0(X,\, E\otimes
{\mathbb L})\,=\, 0$. From the semicontinuity theorem, \cite[p.~288, Theorem 12.8]{Ha}, we know that
${\mathcal M}_X(r)$ is a Zariski open subset of $\widetilde{\mathcal M}_X(r)$. In fact, ${\mathcal M}_X(r)$
is known to be the complement of a theta divisor on $\widetilde{\mathcal M}_X(r)$ \cite{La}. For any
$E\, \in\, \widetilde{\mathcal M}_X(r)$, the Riemann--Roch theorem says
$$
\dim H^0(X,\, E\otimes {\mathbb L}) - \dim H^1(X,\, E\otimes {\mathbb L})\,=\, 0;
$$
so $H^0(X,\, E\otimes {\mathbb L})\,=\,0$ if and only if we have $H^1(X,\, E\otimes {\mathbb L})\,=\,0$.

We will now recall a construction from \cite{BH1}, \cite{BH2}.

For $i\,=\, 1,\, 2$, let $p_i\, :\, X\times X\, \longrightarrow\, X$ be the projection to the
$i$-th factor. Let
$$
\Delta\, :=\, \{(x,\, x)\, \in\, X\times X\,\, \mid\,\, x\, \in\, X\} \, \subset\, X\times X
$$
be the reduced diagonal divisor. We will identify $\Delta$ with $X$ using the map
$x\, \longmapsto\, (x,\, x)$. Using this identification, the restriction of the line bundle
${\mathcal O}_{X\times X}(\Delta)$ to $\Delta\, \subset\, X\times X$ gets identified with
the holomorphic tangent bundle $TX$ by the Poincar\'e adjunction formula \cite[p.~146]{GH}.

Take any $E\,\in\, {\mathcal M}_X(r)$. The restriction of the vector bundle
$$
p^*_1 (E\otimes{\mathbb L})\otimes p^*_2 (E^*\otimes{\mathbb L})\otimes{\mathcal O}_{X
\times X}(\Delta)
$$
to $\Delta$ is identified with the vector bundle $\text{End}(E)$ on $X$. Indeed, this follows
immediately from the following facts:
\begin{itemize}
\item the restriction of $(p^*_1 E)\otimes (p^*_2 E^*)$ to
$\Delta$ is identified with the vector bundle $\text{End}(E)$ on $X$, and

\item the above identification of ${\mathcal O}_{X\times X}(\Delta)\vert_\Delta$ with $TX$ produces an
identification of $(p^*_1 {\mathbb L})\otimes (p^*_2 {\mathbb L})\otimes{\mathcal O}_{X
\times X}(\Delta)\vert_\Delta$ with $K_X\otimes TX\,=\, {\mathcal O}_X$.
\end{itemize}
Consequently, we have the following short exact sequence of sheaves on $X\times X$:
\begin{equation}\label{e2}
0\, \longrightarrow\, p^*_1 (E\otimes{\mathbb L})\otimes p^*_2 (E^*\otimes{\mathbb L})
\longrightarrow\, p^*_1 (E\otimes{\mathbb L})\otimes p^*_2 (E^*\otimes{\mathbb L})\otimes{\mathcal O}_{X
\times X}(\Delta) \, \longrightarrow\, \text{End}(E) \, \longrightarrow\, 0\, ,
\end{equation}
where $\text{End}(E)$ is supported on $\Delta$, using the identification of $\Delta$ with $X$. For $k\,=\, 0,\, 1$,
since $H^k(X,\, E\otimes{\mathbb L})\,=\, 0$, the Serre duality implies that
$H^{1-k}(X,\, E^*\otimes{\mathbb L})\,=\, 0$. Using K\"unneth formula,
$$
H^j(X\times X,\, p^*_1 (E\otimes{\mathbb L})\otimes p^*_2 (E^*\otimes{\mathbb L}))\,=\, 0
$$
for $j\,=\, 0,\, 1,\, 2$. Therefore, the long exact sequence of cohomologies for the short exact
sequence of sheaves in \eqref{e2} gives
$$
0\,=\,
H^0(X\times X,\, p^*_1 (E\otimes{\mathbb L})\otimes p^*_2 (E^*\otimes{\mathbb L}))\, \longrightarrow\,
H^0(X\times X,\, p^*_1 (E\otimes{\mathbb L})\otimes p^*_2 (E^*\otimes{\mathbb L})\otimes{\mathcal O}_{X
\times X}(\Delta))
$$
\begin{equation}\label{e3}
\stackrel{\gamma}\longrightarrow\, H^0(X,\, \text{End}(E)) \,
\longrightarrow\, H^1(X\times X,\, p^*_1 (E\otimes{\mathbb L})\otimes p^*_2 (E^*\otimes{\mathbb L}))\,=\,0\, .
\end{equation}
So the homomorphism $\gamma$ in \eqref{e3} is actually an isomorphism. For this isomorphism $\gamma$, let
\begin{equation}\label{e4}
{\mathcal A}_E\, :=\, \gamma^{-1}({\rm Id}_E)\, \in\, H^0(X\times X,\, p^*_1 (E\otimes{\mathbb L})\otimes p^*_2
(E^*\otimes{\mathbb L})\otimes{\mathcal O}_{X\times X}(\Delta))
\end{equation}
be the section corresponding to the identity automorphism of $E$.

\section{A section around the diagonal}

Using the section ${\mathcal A}_E$ in \eqref{e4} we will construct a section of $(p^*_1 {\mathbb L})\otimes
(p^*_2 {\mathbb L})\otimes{\mathcal O}_{X\times X}(\Delta)$ on an analytic neighborhood of the diagonal
$\Delta$. For that, we first recall a description of the holomorphic differential operators on $X$.

\subsection{Differential operators}

Fix holomorphic vector bundles $V,\, W$ on $X$, and also fix an integer $d\, \geq\, 1$. The ranks of
$V$ and $W$ are denoted by $r$ and $r'$ respectively.
Let $\text{Diff}^d_X(V,\, W)$ denote the holomorphic vector bundle on $X$ of rank $rr'(d+1)$ corresponding to
the sheaf of differential operators of degree $d$ from $V$ to $W$. We recall that
$\text{Diff}^d_X(V,\, W)\,=\, W\otimes J^d(V)^*$, where
$$
J^d(V)\, :=\, p_{1*}\left((p^*_2V)/((p^*_2V)\otimes {\mathcal O}_{X\times X}(-(d+1)\Delta))\right)
\, \longrightarrow\, X
$$
is the $d$-th order jet bundle for $V$.

We have a short exact sequence of coherent analytic sheaves on $X\times X$
\begin{gather}
0\, \longrightarrow\, (p^*_1 W)\otimes p^*_2(V^*\otimes K_X)\, \longrightarrow\,(p^*_1 W)\otimes p^*_2(V^*\otimes K_X)
\otimes {\mathcal O}_{X\times X}((d+1){\Delta})\nonumber\\
\longrightarrow\, {\mathcal Q}_d(V,\,W)\, :=\,
\frac{(p^*_1 W)\otimes p^*_2(V^*\otimes K_X)
\otimes {\mathcal O}_{X\times X}((d+1){\Delta})}{(p^*_1 W)\otimes p^*_2(V^*\otimes K_X)} \, \longrightarrow\, 0\, ;
\label{e5}
\end{gather}
the support of the quotient sheaf ${\mathcal Q}_d(V,\,W)$ in \eqref{e5}
is $(d+1){\Delta}$. The direct image
\begin{equation}\label{e6}
{\mathcal K}_d(V,\,W)\, :=\, p_{1*} {\mathcal Q}_d(V,\,W) \, \longrightarrow\, X
\end{equation}
is a holomorphic vector bundle on $X$ of rank $rr'(d+1)$. It is known that
\begin{equation}\label{e7}
{\mathcal K}_d(V,\,W)\,=\, \text{Hom}(J^d(V),\, W)\,=\, \text{Diff}^d_X(V,\, W)\, ,
\end{equation}
where ${\mathcal K}_d(V,\,W)$ is the vector bundle constructed in \eqref{e6}
(see \cite[Section 2.1]{BS}).

Note that $R^1p_{1*} {\mathcal Q}_d(V,\,W)\,=\, 0$, because ${\mathcal Q}_d(V,\,W)$ is supported
on $(d+1){\Delta}$. We have $H^0(X,\, p_{1*} {\mathcal Q}_d(V,\,W))\,=\, H^0(X\times X,\, {\mathcal Q}_d(V,\,W))$.
So from \eqref{e6} and \eqref{e7} it follows that
\begin{equation}\label{e7a}
H^0(X,\, \text{Diff}^d_X(V,\, W))\, =\, H^0(X\times X,\, {\mathcal Q}_d(V,\,W))\, .
\end{equation}

The restriction of the vector bundle $(p^*_1 W)\otimes p^*_2(V^*\otimes K_X)
\otimes {\mathcal O}_{X\times X}((d+1){\Delta})$ to $\Delta\, \subset\, X\times X$ is
$\text{Hom}(V,\, W)\otimes (TX)^{\otimes d}$, because the restriction of
${\mathcal O}_{X\times X}({\Delta})$ to $\Delta$ is $TX$. Therefore, we get a surjective homomorphism
\begin{gather}
{\mathcal K}_d(V,\,W)\, \longrightarrow\, p_{1*}\left(\frac{(p^*_1 W)\otimes p^*_2(V^*\otimes K_X)
\otimes {\mathcal O}_{X\times X}((d+1){\Delta})}{(p^*_1 W)\otimes p^*_2(V^*\otimes K_X)
\otimes {\mathcal O}_{X\times X}(d{\Delta})}\right)\nonumber\\
\,=\, \text{Hom}(V,\, W)\otimes (TX)^{\otimes d}\, ,\nonumber
\end{gather}
where ${\mathcal K}_d(V,\,W)$ is constructed in \eqref{e6}.
Using \eqref{e7}, this gives a surjective homomorphism
\begin{equation}\label{e8}
\text{Diff}^d_X(V,\, W)\, \longrightarrow\,\text{Hom}(V,\, W)\otimes (TX)^{\otimes d}\, .
\end{equation}
The homomorphism in \eqref{e8} is known as the \textit{symbol map}.

\subsection{Construction of a connection}

Consider the de Rham differential operator $d\, :\, {\mathcal O}_X\, \longrightarrow\, K_X$. Using the
isomorphism in \eqref{e7a} this $d$ produces a section
$$d_1\, \in\, H^0(X\times X,\, {\mathcal Q}_1({\mathcal O}_X,\,K_X)).
$$
{}From \eqref{e5} we conclude that $d_1$ is a section of
$(p^*_1 K_X)\otimes (p^*_2 K_X)\otimes {\mathcal O}_{X\times X}(2{\Delta})$ over
$2{\Delta}$. The restriction of $d_1$ to $\Delta\, \subset\,2\Delta$ is the section of ${\mathcal O}_X$
given by the constant function $1$ (see \eqref{e8}); note that the symbol of the differential operator $d$ is the
constant function $1$.

As before, ${\mathbb L}$ is a theta characteristic on $X$. There is a unique section
\begin{equation}\label{e9}
\delta\,\in\, H^0(2{\Delta},\,(p^*_1{\mathbb L})\otimes (p^*_2{\mathbb L})\otimes {\mathcal O}_{X\times X}({\Delta}))
\end{equation}
such that
\begin{enumerate}
\item $\delta\otimes\delta\,=\, d_1$, and

\item the restriction of $\delta$ to $\Delta\, \subset\,2\Delta$ is the constant function $1$
(note that since the restriction of ${\mathcal O}_{X\times X}({\Delta})$ to $\Delta$ is $TX$, the
restriction of $(p^*_1{\mathbb L})\otimes (p^*_2{\mathbb L})\otimes {\mathcal O}_{X\times X}({\Delta})$
to $\Delta$ is $K_X\otimes TX\,=\, {\mathcal O}_X$).
\end{enumerate}
See \cite[p.~754, Theorem 2.1(b)]{BR} for an alternative construction of $\delta$.

There is a unique section
\begin{equation}\label{e10}
\Phi_E\,\in\, H^0(2{\Delta},\, (p^*_1 E)\otimes (p^*_2 E^*))
\end{equation}
such that $({\mathcal A}_E)\big\vert_{2\Delta}\, =\, \Phi_E\otimes \delta$, where ${\mathcal A}_E$ and $\delta$
are the sections in \eqref{e4} and \eqref{e9} respectively. Indeed, this follows immediately from the
fact that the section $\delta$ is nowhere zero, so $\delta^{-1}$ is a holomorphic section of $((p^*_1{\mathbb L})
\otimes (p^*_2{\mathbb L})\otimes {\mathcal O}_{X\times X}({\Delta}))^*\big\vert_{2\Delta}$. Now set
$$
\Phi_E\,=\, (({\mathcal A}_E)\big\vert_{2\Delta})\otimes \delta^{-1}
$$
and consider it as a section of $((p^*_1 E)\otimes (p^*_2 E^*))\big\vert_{2\Delta}$ using the natural
duality pairing
$$
((p^*_1{\mathbb L})
\otimes (p^*_2{\mathbb L})\otimes {\mathcal O}_{X\times X}({\Delta}))\big\vert_{2\Delta}\otimes 
((p^*_1{\mathbb L})
\otimes (p^*_2{\mathbb L})\otimes {\mathcal O}_{X\times X}({\Delta}))^*\big\vert_{2\Delta}
\, \longrightarrow\, {\mathcal O}_{2\Delta}\,.
$$

Since the restriction of ${\mathcal A}_E$ to $\Delta$ is ${\rm Id}_E$ (see \eqref{e4}), and the 
restriction of $\delta$ to $\Delta$ is the constant function $1$, it follows that the
restriction of the section $\Phi_E$ in \eqref{e10} to $\Delta$ is ${\rm Id}_E$. Therefore, $\Phi_E$
defines a holomorphic connection on $U$, which will be denoted by $D^E$. To describe the connection $D^E$
explicitly, take an open subset $U\, \subset\, X$ and a holomorphic section
$s\, \in\, H^0(U,\, E\big\vert_U)$. Consider the section $\Phi_E\otimes p^*_2 s$ over ${\mathcal U}\, :=\,
(2\Delta)\bigcap (U\times U)$. Using the natural pairing $E^*\otimes E\, \longrightarrow\, {\mathcal O}_X$, it produces
a section of $((p^*_1 E)\otimes (p^*_2{\mathcal O}_X))\big\vert_{\mathcal U}
\,=\,(p^*_1 E)\big\vert_{\mathcal U}$; denote this section of $(p^*_1 E)\big\vert_{\mathcal U}$ by $\widetilde{s}$. Since
$\Phi_E\big\vert_{\Delta}\,=\, {\rm Id}_E$, we know that $\widetilde{s}$ and $p^*_1 s$ coincide on
$\Delta\bigcap (U\times U)$. So
$$
\widetilde{s}- (p^*_1 s)\big\vert_{2\Delta}\, \in\, H^0(U,\, E\otimes K_X)\, ;
$$
the Poincar\'e adjunction formula identifies $K_X$ with the restriction of the line bundle
${\mathcal O}_{X\times X}(-\Delta)$ to $\Delta$. Then we have
\begin{equation}\label{e11}
D^E (s) \,=\, \widetilde{s}- (p^*_1 s)\big\vert_{2\Delta} \, \in\, H^0(U,\, E\otimes K_X)\, .
\end{equation}
It is straightforward to check that $D^E$ satisfies the Leibniz identity thus making it a holomorphic
connection on $E$.

\section{Opers from vector bundles}

We will construct an ${\rm SL}(n,{\mathbb C})$-oper on $X$, for every $n\, \geq\, 2$,
from the section ${\mathcal A}_E$ in \eqref{e4}.

We will use that any holomorphic connection on a holomorphic bundle over $X$ is integrable (same as flat)
because $\Omega^2_X\,=\, 0$.

As before, take any $E\, \in\, {\mathcal M}_X(r)$. Consider the holomorphic connection $D^E$ on
$E$ in \eqref{e11}. Let $U\, \subset\, X$ be a simply connected open subset, and let $x_0\, \in\, U$
be a point. Since the connection $D^E$ is integrable, using parallel translations, for $D^E$, along paths
emanating from $x_0$ we get a holomorphic isomorphism of $E\big\vert_U$ with the trivial vector bundle
$U\times E_{x_0}\, \longrightarrow\, U$. This isomorphism takes the connection $D^E\big\vert_U$
on $E\big\vert_U$ to the trivial connection on the trivial bundle. Let
$$
\Delta\, \subset\, {\mathcal U}\, \subset\, X\times X
$$
be an open neighborhood of $\Delta$ that admits a deformation retraction to $\Delta$. For $i\,=\,1,\, 2$,
the restriction of the projection $p_i\, :\, X\times X\,\longrightarrow\, X$ to the open subset
${\mathcal U}\, \subset\, X\times X$ will be denoted by $q_i$.
There is a unique holomorphic isomorphism over ${\mathcal U}$
\begin{equation}\label{e12}
f \, :\, q^*_1 E\, \longrightarrow\, q^*_2 E
\end{equation}
that satisfies the following two conditions:
\begin{enumerate}
\item the restriction of $f$ to $\Delta$ is the identity map of $E$, and

\item $f$ takes the connection $q^*_1 D^E$ on $q^*_1 E$ to the connection $q^*_2 D^E$ on $q^*_2 E$.
\end{enumerate}
Since the inclusion map $\Delta\, \hookrightarrow\, {\mathcal U}$ is a homotopy equivalence, the flat
vector bundle $(E,\, D^E)$ on $X\,=\,\Delta$ has a unique extension to a flat vector bundle on
${\mathcal U}$. On the other hand, both $(q^*_1 E,\, q^*_1 D^E)$ and $(q^*_2 E,\, q^*_2 D^E)$ are extensions of 
$(E,\, D^E)$. Therefore, there is a unique holomorphic isomorphism $f$ as in \eqref{e12}
satisfying the above two conditions.

Using the isomorphism $f$ in \eqref{e12}, the section ${\mathcal A}_E$ in \eqref{e4} produces a
holomorphic section 
\begin{equation}\label{e13}
{\mathcal A}'_E \, \in\, H^0({\mathcal U},\, p^*_1 (E\otimes E^*\otimes{\mathbb L})\otimes (p^*_2
{\mathbb L})\otimes{\mathcal O}_{X\times X}(\Delta))\, .
\end{equation}
Consider the trace pairing
$$
tr\,\, :\,\, E\otimes E^*\,=\, \text{End}(E) \, \longrightarrow\, {\mathcal O}_X\, , \ \ B\, \longmapsto\,
\frac{1}{r}\text{trace}(B);
$$
recall that $r\,=\, \text{rank}(E)$. Note that $r\cdot tr$ is the natural pairing
$E\otimes E^*\, \longrightarrow\, {\mathcal O}_X$. Using $tr$, the section ${\mathcal A}'_E$ in \eqref{e13}
produces a section
\begin{equation}\label{e14}
\widehat{\beta}_E\, \in\, H^0({\mathcal U},\, (p^*_1 {\mathbb L})\otimes (p^*_2
{\mathbb L})\otimes{\mathcal O}_{X\times X}(\Delta))\, .
\end{equation}

The following lemma is straightforward to prove.

\begin{lemma}\label{lem1}
The restriction of the section $\widehat{\beta}_E$ (in \eqref{e14}) to $2\Delta\, \subset\,
{\mathcal U}$ coincides with the section $\delta$ in \eqref{e9}.
\end{lemma}

\begin{proof}
This follows from the constructions of $\Phi_E$ (in \eqref{e10}) and $\widehat{\beta}_E$.
\end{proof}

For any integer $k\, \geq\, 1$, the holomorphic line bundles ${\mathbb L}^{\otimes k}$
and $({\mathbb L}^{\otimes k})^*$ will be denoted by ${\mathbb L}^k$ and ${\mathbb L}^{-k}$ respectively.

\begin{theorem}\label{thm1}
Take any integer $n\, \geq\, 2$. The section $\widehat{\beta}_E$ in \eqref{e14} produces a holomorphic
connection ${\mathcal D}(E)$ on the holomorphic vector bundle $J^{n-1}({\mathbb L}^{(1-n)})$.
\end{theorem}

\begin{proof}
Consider the $(n+1)$-th tensor power of $\widehat{\beta}_E$
$$
(\widehat{\beta}_E)^{\otimes (n+1)}\, \in\,
H^0({\mathcal U},\, (p^*_1 {\mathbb L}^{(n+1)})\otimes (p^*_2
{\mathbb L}^{(n+1)})\otimes{\mathcal O}_{X\times X}((n+1)\Delta)),
$$
and restrict it to $(n+1)\Delta\, \subset\, {\mathcal U}$. From \eqref{e7a} we have
\begin{equation}\label{e15}
\beta^n_E \, :=\,(\widehat{\beta}_E)^{\otimes (n+1)}\big\vert_{(n+1)\Delta}\, \in\,
H^0(X,\, \text{Diff}^n_X({\mathbb L}^{(1-n)},\, {\mathbb L}^{(n+1)}))\, .
\end{equation}
The symbol of the differential operator $\beta^n_E$ in \eqref{e15} is the section of ${\mathcal O}_X$ given by the
constant function $1$. Indeed, this follows immediately from the fact that the restriction of
$\widehat{\beta}_E$ to $\Delta$ is the constant function $1$ (see Lemma \ref{lem1}).

We recall that there is a natural injective homomorphism $J^{m+n}(V)\, \longrightarrow\,
J^m(J^n(V))$ for all $m,\, n\, \geq\, 0$ and every holomorphic vector bundle $V$.
We have following commutative diagram of vector bundles on $X$
\begin{equation}\label{e16}
\begin{matrix}
0 &\longrightarrow & {\mathbb L}^{(1-n)}\otimes K^{\otimes n}_X\,=\,
{\mathbb L}^{(n+1)} &\stackrel{\iota_1}{\longrightarrow} & J^n({\mathbb L}^{(1-n)})
&\stackrel{\psi_1}{\longrightarrow} & J^{n-1}({\mathbb L}^{(1-n)}) &\longrightarrow & 0\\
&& \Big\downarrow && \,\,\, \Big\downarrow \mathbf{h} && \Big\Vert \\
0 &\longrightarrow & J^{n-1}({\mathbb L}^{(1-n)})\otimes K_X & \stackrel{\iota_2}{\longrightarrow} &
J^1(J^{n-1}({\mathbb L}^{(1-n)}))&\stackrel{\psi_2}{\longrightarrow} & J^{n-1}({\mathbb L}^{(1-n)}) &\longrightarrow & 0
\end{matrix}
\end{equation}
where the rows are exact. From \eqref{e7} we know that the differential operator $\beta^n_E$
in \eqref{e15} produces a homomorphism
$$
\rho\, :\, J^n({\mathbb L}^{(1-n)})\, \longrightarrow\, {\mathbb L}^{(n+1)}\, .
$$
Since the symbol of $\beta^n_E$ is the constant function $1$, we have
\begin{equation}\label{e17}
\rho\circ\iota_1\,=\, {\rm Id}_{{\mathbb L}^{(n+1)}}\, ,
\end{equation}
where $\iota_1$ is the homomorphism in \eqref{e16}. From \eqref{e17} it follows immediately that $\rho$
gives a holomorphic splitting of the top exact sequence in \eqref{e16}. Let
$$
{\mathcal D}_1\, :\, J^{n-1}({\mathbb L}^{(1-n)})\, \longrightarrow\, J^n({\mathbb L}^{(1-n)})
$$
be the unique holomorphic homomorphism such that
\begin{itemize}
\item $\rho\circ {\mathcal D}_1\,=\, 0$, and

\item $\psi_1\circ {\mathcal D}_1\,=\, {\rm Id}_{J^{n-1}({\mathbb L}^{(1-n)})}$, where $\psi_1$ is
the projection in \eqref{e16}.
\end{itemize}
Now consider the homomorphism
\begin{equation}\label{e18}
{\mathcal D}_2\, :=\, {\mathbf h}\circ {\mathcal D}_1\, :\, J^{n-1}({\mathbb L}^{(1-n)})\, \longrightarrow\,
J^1(J^{n-1}({\mathbb L}^{(1-n)}))\, ,
\end{equation}
where $\mathbf h$ is the homomorphism in \eqref{e16}. Since $\psi_1\circ {\mathcal D}_1\,=\,
{\rm Id}_{J^{n-1}({\mathbb L}^{(1-n)})}$, from the commutativity of \eqref{e16} it follows that
$$
\psi_2\circ {\mathcal D}_2\,=\, \psi_2\circ {\mathbf h}\circ {\mathcal D}_1\,=\, {\rm Id}_{J^{n-1}({\mathbb L}^{(1-n)})}
\circ\psi_1\circ {\mathcal D}_1\,=\, {\rm Id}_{J^{n-1}({\mathbb L}^{(1-n)})}\, ,
$$
where $\psi_2$ is the projection in \eqref{e16}. This implies that ${\mathcal D}_2$ in \eqref{e18}
gives a holomorphic splitting of the bottom exact sequence in \eqref{e16}.

Let
$$
{\mathcal D}(E)\, :\, J^1(J^{n-1}({\mathbb L}^{(1-n)}))\, \longrightarrow\,
J^{n-1}({\mathbb L}^{(1-n)})\otimes K_X
$$
be the unique holomorphic homomorphism such that
\begin{itemize}
\item ${\mathcal D}(E)\circ {\mathcal D}_2\,=\, 0$, and

\item ${\mathcal D}(E)\circ\iota_2\,=\, {\rm Id}_{J^{n-1}({\mathbb L}^{(1-n)})\otimes K_X}$, where
$\iota_2$ is the homomorphism in \eqref{e16}.
\end{itemize}
Using \eqref{e7} we know that
\begin{equation}\label{l1}
{\mathcal D}(E)\, \in\, H^0(X,\,
\text{Diff}^1_X(J^{n-1}({\mathbb L}^{(1-n)}),\, J^{n-1}({\mathbb L}^{(1-n)})\otimes K_X))\, ;
\end{equation}
from the homomorphism in \eqref{e8} and the above equality ${\mathcal D}(E)\circ\iota_2\,=\, {\rm Id}_{J^{n-1}
({\mathbb L}^{(1-n)})\otimes K_X}$ it follows that the symbol of the differential operator
${\mathcal D}(E)$ in \eqref{l1} is ${\rm Id}_{J^{n-1}({\mathbb L}^{(1-n)})}$. This implies that
${\mathcal D}(E)$ satisfies the Leibniz rule. Consequently, ${\mathcal D}(E)$ is a holomorphic connection
on the holomorphic vector bundle $J^{n-1}({\mathbb L}^{(1-n)})$.
\end{proof}

For $1\, \leq\, i\, \leq\, n-1$, consider the short exact sequence
$$
0 \, \longrightarrow\, {\mathbb L}^{(1-n)}\otimes K^{\otimes i}_X\, \longrightarrow\, J^{i}({\mathbb L}^{(1-n)})
\, \longrightarrow\, J^{i-1}({\mathbb L}^{(1-n)})\, \longrightarrow\, 0\, .
$$
Using these together with the fact that ${\mathbb L}\otimes {\mathbb L}\,=\, K_X$ it is deduced that
$$
\det J^i({\mathbb L}^{(1-n)})\,:=\, \bigwedge\nolimits^{i+1} J^i({\mathbb L}^{(1-n)})\,=\,
{\mathbb L}^{(i+1)(i+1-n)}\, .
$$
In particular, we have
$$
\det J^{n-1}({\mathbb L}^{(1-n)})\,:=\, \bigwedge\nolimits^n J^{n-1}({\mathbb L}^{(1-n)})\,=\, {\mathcal O}_X\, .
$$
So $\det J^{n-1}({\mathbb L}^{(1-n)})$ has a unique holomorphic connection whose monodromy is trivial;
it will be called the trivial connection on $\det J^{n-1}({\mathbb L}^{(1-n)})$.

\begin{proposition}\label{prop1}
The holomorphic connection on $\det J^{n-1}({\mathbb L}^{(1-n)})$ induced by the connection ${\mathcal D}(E)$
on $J^{n-1}({\mathbb L}^{(1-n)})$ (see Theorem \ref{thm1}) is the trivial connection.
\end{proposition}

\begin{proof}
Any holomorphic connection $D$ on $\det J^{n-1}({\mathbb L}^{(1-n)})\,=\, {\mathcal O}_X$ can be uniquely
expressed as
$$
D\,=\, D_0+\omega\, ,
$$
where $D_0$ is the trivial connection on $\det J^{n-1}({\mathbb L}^{(1-n)})$ and
$\omega\, \in\, H^0(X,\, K_X)$. Let $D^1$ be the holomorphic connection on $\det J^{n-1}({\mathbb L}^{(1-n)})$
induced by the connection ${\mathcal D}(E)$ on $J^{n-1}({\mathbb L}^{(1-n)})$. Decompose it as
$$
D^1\,=\, D_0+\omega^1\, ,
$$
where $\omega^1\, \in\, H^0(X,\, K_X)$. Then
$$
\omega^1\,=\, (n-1)\cdot \left((\widehat{\beta}_E)\big\vert_{2\Delta}-\delta\right)\, ,
$$
where $\widehat{\beta}_E$ and $\delta$ are the sections in \eqref{e14} and \eqref{e9} respectively.
Note that two sections of $((p^*_1{\mathbb L})\otimes (p^*_2{\mathbb L})\otimes
{\mathcal O}_{X\times X}({\Delta}))\big\vert_{2\Delta}$ that coincide on $\Delta\, \subset\, 2\Delta$
differ by an element of $H^0(\Delta,\, ((p^*_1{\mathbb L})\otimes (p^*_2{\mathbb L}))\big\vert_{\Delta})
\,=\,H^0(X,\, K_X)$. Now from Lemma \ref{lem1} it follows that $\omega^1\,=\, 0$.
\end{proof}

Let ${\rm Op}_X(n)$ denote the moduli space of ${\rm SL}(n, {\mathbb C})$ opers on $X$ \cite{BD2}.
It is a complex affine space of dimension $(n^2-1)(g-1)$. We recall 
a description of ${\rm Op}_X(n)$. Let ${\mathcal C}_n(X)$ denote the space of all holomorphic connections
$D'$ on $J^{n-1}({\mathbb L}^{(1-n)})$ such that the holomorphic connection on
$\det J^{n-1}({\mathbb L}^{(1-n)})$ induced by $D'$ is the trivial connection on
$\det J^{n-1}({\mathbb L}^{(1-n)})\,=\, {\mathcal O}_X$. Then
$$
{\rm Op}_X(n)\,=\, {\mathcal C}_n(X)/{\rm Aut}(J^{n-1}({\mathbb L}^{(1-n)}))\, ,
$$
where ${\rm Aut}(J^{n-1}({\mathbb L}^{(1-n)}))$ denotes the group of all holomorphic automorphisms of
the vector bundle $J^{n-1}({\mathbb L}^{(1-n)})$; note that ${\rm Aut}(J^{n-1}({\mathbb L}^{(1-n)}))$ has a natural
action on ${\mathcal C}_n(X)$.

The moduli space ${\rm Op}_X(n)$ also coincides with the space of all
holomorphic differential operators 
$$
{\mathcal B}\, \in\, H^0(X,\, \text{Diff}^n_X({\mathbb L}^{(1-n)},\, {\mathbb L}^{(n+1)}))
$$
such that
\begin{enumerate}
\item the symbol of ${\mathcal B}$ is the section of ${\mathcal O}_X$ given by the constant function $1$, and

\item the sub-leading term of ${\mathcal B}$ vanishes.
\end{enumerate}

From Theorem \ref{thm1} and and Proposition \ref{prop1} we get an algebraic morphism
\begin{equation}\label{tPsi}
\widetilde{\Psi}\, :\, {\mathcal M}_X(r)\,\longrightarrow\, {\rm Op}_X(n)
\end{equation}
that sends any $E\, \in\, {\mathcal M}_X(r)$ to the image, in ${\rm Op}_X(n)$, of the
holomorphic connection ${\mathcal D}(E)$ (see Theorem \ref{thm1}).

Since $H^i(X,\, E\otimes {\mathbb L})\,=\, H^{1-i}(X,\, E^*\otimes {\mathbb L})$ (Serre duality), we have
an involution
\begin{equation}\label{ei}
{\mathcal I}\, :\,{\mathcal M}_X(r)\, \longrightarrow\, {\mathcal M}_X(r)\, ,\ \ F\, \longmapsto\, F^*\, .
\end{equation}

Let
\begin{equation}\label{et}
\tau\, :\, X\times X\, \longrightarrow\,X\times X\, ,\ \ (x_1,\, x_2)\, \longmapsto\, (x_2,\, x_1)
\end{equation}
be the involution.

\begin{proposition}\label{prop2}
For any $E\, \in\, {\mathcal M}_X(r)$, the sections
$$
{\mathcal A}_E\, \in\, H^0(X\times X,\, p^*_1 (E\otimes{\mathbb L})\otimes p^*_2
(E^*\otimes{\mathbb L})\otimes{\mathcal O}_{X\times X}(\Delta))
$$
and
$$
{\mathcal A}_{{\mathcal I}(E)}\, \in\, H^0(X\times X,\, p^*_1 (E^*\otimes{\mathbb L})\otimes p^*_2
(E\otimes{\mathbb L})\otimes{\mathcal O}_{X\times X}(\Delta))
$$
(see \eqref{e4} for ${\mathcal A}_E$ and \eqref{ei} for ${\mathcal I}$) satisfy the equation
$$
\tau^*{\mathcal A}_E\,=\, {\mathcal A}_{{\mathcal I}(E)}\, ,
$$
where $\tau$ is the involution in \eqref{et}.
\end{proposition}

\begin{proof}
We recall that ${\mathcal A}_E$ is the unique section of $p^*_1 (E^*\otimes{\mathbb L})\otimes p^*_2
(E\otimes{\mathbb L})\otimes{\mathcal O}_{X\times X}(\Delta)$ over $X\times X$ whose restriction to
$\Delta$ coincides with ${\rm Id}_E$ using the identification of $\Delta$ with $X$. Now the
restriction of $\tau^*{\mathcal A}_{{\mathcal I}(E)}$ to $\Delta$ is also ${\rm Id}_E$. So,
$\tau^*{\mathcal A}_{{\mathcal I}(E)}\,=\, {\mathcal A}_E$, which implies that
$\tau^*{\mathcal A}_E\,=\, {\mathcal A}_{{\mathcal I}(E)}$.
\end{proof}

\begin{corollary}\label{cor1}
The map $\widetilde{\Psi}$ in \eqref{tPsi} descends to a map
$$
\widetilde{\Psi}^0\, :\, {\mathcal M}_X(r)/{\mathcal I}\,\longrightarrow\,{\rm Op}_X(n)\, ,
$$
where $\mathcal I$ is the involution in \eqref{ei}.
\end{corollary}

Let $\xi$ be a holomorphic line bundle on $X$ such that $\xi\otimes\xi\,=\, {\mathcal O}_X$; for example,
$\xi$ can be ${\mathcal O}_X$. Let
$$
{\mathcal M}_X(r,\xi)\, \subset\, {\mathcal M}_X(r)
$$
be the sub-variety consisting of all $E\, \in\, {\mathcal M}_X(r)$ such that $\bigwedge^r E\,=\, \xi$.
Since $\xi^{\otimes 2}\,=\, {\mathcal O}_X$ it follows that
$$
{\mathcal I}({\mathcal M}_X(r,\xi))\,=\, {\mathcal M}_X(r,\xi)\, ,
$$
where $\mathcal I$ is defined in \eqref{ei}. So restricting the map $\widetilde{\Psi}^0$ in
Corollary \ref{cor1} to ${\mathcal M}_X(r,\xi)/{\mathcal I}$ we get morphism
\begin{equation}\label{ep}
\Psi\, :\, {\mathcal M}_X(r,\xi)/{\mathcal I}\,\longrightarrow\, {\rm Op}_X(n)\, .
\end{equation}

We note that when $n\,=\, r$,
$$
\dim {\mathcal M}_X(r,\xi)/{\mathcal I}\,=\, (r^2-1)(g-1)\,=\, \dim {\rm Op}_X(r)\, .
$$

So it is natural to ask the following:

\begin{question}\label{q1}
When $n\,=\, r$, 
how close is the map $\Psi$ (constructed in \eqref{ep}) to being injective or surjective?
\end{question}

\section*{Data Availability}

Data sharing not applicable -- no new data generated.



\begin{thebibliography}{AAAA}

\bibitem[At]{At} M. F. Atiyah, Complex analytic connections in fibre
bundles, \textit{Trans. Amer. Math. Soc.} \textbf{85} (1957), 181--207.

\bibitem[BD1]{BD1}
A.~Beilinson and V. G. Drinfeld, Quantization of Hitchin's integrable system and Hecke eigensheaves, (1991).

\bibitem[BD2]{BD2} A. Beilinson and V. G. Drinfeld, Opers,
{\em arXiv math/0501398} (1993).

\bibitem[BS]{BS} A. A. Beilinson and V. V. Schechtman, Determinant bundles and Virasoro
algebras, {\it Comm. Math. Phys.} {\bf 118} (1988), 651--701.

\bibitem[BB]{BB} D. Ben-Zvi and I. Biswas, Theta functions and Szeg\H{o} kernels,
{\it Int. Math. Res. Not.} (2003), no. 24, 1305--1340. 

\bibitem[BH1]{BH1} I. Biswas and J. Hurtubise, Meromorphic
connections, determinant line bundles and the Tyurin parametrization, arXiv:1907.00133,
\textit{Asian Jour. Math.} (to appear).

\bibitem[BH2]{BH2} I. Biswas and J. Hurtubise, A canonical connection on bundles on 
Riemann surfaces and Quillen connection on the theta bundle, {\it Adv. Math.}
{\bf 389} (2021), Paper No. 107918.

\bibitem[BR]{BR} I. Biswas and A. K. Raina, Projective structures on a Riemann surface, {\it
Inter. Math. Res. Not.} (1996), No. 15, 753--768. 

\bibitem[GH]{GH} P. Griffiths and J. Harris, {\it Principles of algebraic geometry},
Pure and Applied Mathematics, Wiley-Interscience, New York, 1978.

\bibitem[Ha]{Ha} R. Hartshorne, {\it Algebraic geometry}, Graduate Texts in
Mathematics, No. 52. Springer-Verlag, New York-Heidelberg, 1977.

\bibitem[La]{La} Y. Laszlo, Un th\'eor\`eme de Riemann pour les diviseurs th\^eta sur les espaces
de modules de fibr\'es stables sur une courbe, {\it Duke Math. Jour.} {\bf 64} (1991), 333--347. 

\end{thebibliography}
\end{document}